\documentclass[12pt, reqno,oneside]{amsart}
\usepackage{blindtext}
\usepackage[a4paper, total={6in, 9in}]{geometry}
\usepackage{hyperref}
\usepackage{amsmath}
\usepackage{mathrsfs}
\usepackage{graphicx} 
\usepackage{amsfonts}
\usepackage{amsthm}
\usepackage{graphicx} 
\AtEndDocument{%
	\par
	\medskip
	
\textsc{Department Of Mathematics, Middle East Technical University, 06800 Ankara, Turkey}

\textit{ E-mail address}: \texttt{akyart@metu.edu.tr}
}
\usepackage{thmtools}

\newcommand{\R}{\mathbb{R}}
\newcommand{\C}{\mathbb{C}}
\newcommand{\WRDX}{W_d^r(X_\C)}
\newcommand{\p}{{\mathbb{P}^1}}
\theoremstyle{definition}
\newtheorem{definition}{Definition}
\newtheorem{remark}{Remark}
\newtheorem{example}{Example}

\newcommand{\teotext}{		Let $X$ be a real trigonal curve of genus $g$ and Maroni invariant $m$. If $d$ and $r$ are positive integers such that  $d<g$ and $g-d+r-1=m\leq d-2r-1$, then $$\displaystyle s_{n(X)}\bigl(d-3r\bigr)\leq n(W_d^r)\leq s_{n(X)} \bigl(d-3r\bigr)+\binom{n(X)-1}{2d-g-3r+1},$$
	where $s_n(l)=\displaystyle\sum_{k=0}^{\lfloor l/2\rfloor}\binom{n}{l-2k}$ for any positive integers $n$ and $l$. }
\newtheorem*{T1}{Theorem~\ref{teoremim}}

\newcommand{\thmtext}{	In addition to the conditions of Theorem \ref{teoremim} if we assume $2g-3d+3r=0$, then we have 
	$$n(W_{d}^r)=s_{n(X)}\left(d-3r\right)+\binom{n(X)-1}{2d-g-3r+1}.$$
}
\newtheorem*{T2}{Theorem~\ref{teoremim2}}

\theoremstyle{plain}
\newtheorem{prop}{Proposition}
\newtheorem{teorem}{Theorem}
\newtheorem{lemma}{Lemma}
\newtheorem{corollary}{Corollary}

\begin{document}
	\author{Turgay Akyar}
\title{Special Divisors on Real Trigonal Curves}
\date{}
\maketitle

\begin{abstract}
	In this paper we examine the topology of Brill-Noether varieties associated to real trigonal curves. More precisely, we aim to count the connected components of the real locus of the varieties parametrizing linear systems of degree $d$ and dimension at least $r$. We do this count when the relations $m=g-d+r-1\leq d-2r-1$ are satisfied, where $m$ is the Maroni invariant and $g$ is the genus of the curve.
\end{abstract}
\let\thefootnote\relax\footnotetext{Key words: Real Brill-Noether theory, trigonal curves.}
\let\thefootnote\relax\footnotetext{AMS Subject Classification: 14H51, 14P25.}
\section{Introduction}
Let $X$ be a smooth projective geometrically irreducible curve defined over $\R$. By $X_{\C}$ we denote the complexification of $X$ obtained by base change from $\R$ to $\C$. We want to understand under which conditions real complete linear systems on $X$ exist, or equivalently the real locus of the variety $$W_d^r(X_{\C})=\bigl\{|D|: h^0(X,\mathcal{O}_X(D))\geq r+1,\text{deg}(D)=d\bigr\}\subset Pic^d(X)$$ is non-empty. Indeed, not only the existence of its real points but also topological properties of this real locus have yet to be investigated. One can find \textit{real Brill-Noether} results in \cite{Chaudhary} concerning the existence of general real curves having minimal degree maps to $\mathbb{P}^1_{\R}$. Instead of these maps, in \cite{Huisman} base point-free pencils of prescribed topological degrees are considered.


We will follow a direct approach to the problem.
Let $n(Y)$ stand for the number of connected components of the real locus $Y(\R)$ for any variety $Y$ defined over $\mathbb{R}$. Then we know by Harnack's theorem for curves that $n(X)\leq g+1$. Analogously, a natural topological question for the real Brill-Noether theory is calculating $n(\WRDX)$. The answer for the case $r=0$ is already known for any curve \cite{Harris-Gross}.
 
 When $X$ is trigonal, $\WRDX$ contains at most two irreducible components. Furthermore, each component is a translate of some $W^0_{d'}(X_{\C})$ when we consider them inside the Jacobian $J(X)$ of the curve. This helps us to make use of the case $r=0$ in search of $n(\WRDX).$

 
 The canonical image $\phi_{K}(X)\subset \mathbb{P}^{g-1}$ of $X$, where $\phi_{K}$ is the morphism induced by the canonical bundle $K_X$ of $X$, lies on a Hirzebruch surface \cite[Section 1]{MarSch}  $$Y=\Sigma_{g-2-2m}\cong \mathbb{P}\bigr(\mathcal{O}_\p (g-2-m)\oplus\mathcal{O}_\p(m)\bigl)$$
 which is also a real rational scroll  in $\mathbb{P}^{g-1}$. This unique integer $m$ is called the Maroni invariant of trigonal curve $\pi:X\rightarrow \mathbb{P}^1_\mathbb{C}$ and it satisfies the inequalities $$\dfrac{g-4}{3}\leq m\leq \dfrac{g-2}{2}$$ such that 
  $$\pi_* K_X\cong\mathcal{O}_{\mathbb{P}^1}(g-2-m)\oplus \mathcal{O}_{\mathbb{P}^1}(m)\oplus \mathcal{O}_{\mathbb{P}^1}(-2).$$
 Using the descriptions of connected components of $W^0_{d}(X_{\C})$ and the cohomological calculations obtained by the above isomorphism, we obtain the following result.
\begin{T1}
	\teotext
\end{T1}

If $T$ stands for the unique pencil of degree $3$ on $X_{\C}$, then the number of base points of $K_X-(g-d+2r-1)T$ is $2g-3d+3r+1$. Putting an extra condition on this number, we obtain an exact count.

\begin{T2}
	\thmtext
\end{T2}


\section{Preliminaries}

In this section we recall basic definitions and facts about real linear systems, the varieties parametrizing them and trigonal curves.

\subsection*{Notations}
Throughout this paper  $X$ is a real, smooth, projective, geometrically irreducible curve of genus $g$ together with an  anti-holomorphic involution $\sigma$ induced by complex conjugation. We assume that $X$ has real points, i.e $X(\mathbb{R})\neq\emptyset$. $K_X$ is the canonical divisor class of $X$ or any effective divisor from this class if there is no confusion. If $Y$ is any variety defined over $\R$, then $n(Y)$ represents the number of connected components of $Y(\R)$.

\subsection{Real Linear Systems}

The action of $\sigma$ on $X_\C$ can be linearly extended to the divisor group Div$(X_\C)$. A divisor $D$ on $X_\mathbb{C}$ is called \textit{real} if it is invariant under the $\sigma$-action. Such a divisor will be also called a divisor on $X$ and they all together form the real divisor group $\text{Div}(X)$ of $X$.

If $f\in \C(X)^*$ and $z$ is any local coordinate, then we let $\sigma\bigl(f(z)\bigr)=\overline{f(\overline{z})}$ so that $\sigma$ acts on the principal divisor div$(f)$ by $\sigma\bigl(\text{div}(f)\bigr)=\text{div}\bigl(\sigma(f)\bigr)$. This allows us to extend the action to linear systems as well. Recall that a (complete) $g_d^r$ on $X_\mathbb{C}$ is a linear system $|D|$ such that $deg(D)=d$ and $r(D):=h^0(X_\mathbb{C},\mathcal{O}_{X_\mathbb{C}}(D))-1=r$ is its projective dimension. It is called $\sigma$\textit{-invariant}
if $\sigma(E)\sim E$ for every $E\in |D|$.

Given a degree $d$ real divisor $D$ such that $r(D)=r$, a \textit{real} $g_d^r$ is the set $\{E\in \text{Div}(X): E\sim D\ \text{ and } E \text{ is effective}\}$. Clearly a real $g_d^r$ defines a  $\sigma$-invariant $g_d^r$ on $X_\mathbb{C}$. The converse also holds. By the assumption $X(\mathbb{R})\neq\emptyset$, we know that any $\sigma$-invariant $g_d^r$ can be represented by a real divisor so it defines a unique real $g_d^r$ \cite{Harris-Gross}. Therefore we will not care about the difference and we will use the notion of real $g_d^r$. The canonical class $K_X$, which is also real, is the unique $g_{2g-2}^{g-1}$ on $X_\C$.

\begin{definition}[\cite{Huisnonspe}, Page 23]
	For a real divisor $\mathit{D}=\displaystyle\sum_{p\in X} n_p [p]$, we define $\delta(D)$ to be the number of connected components $C$ of $X(\mathbb{R})$ such that $deg_C(D):=\displaystyle\sum_{p\in C} n_p$ is odd.
\end{definition}

 In \cite[Proposition 2.1]{Huisnonspe} the author also shows that $\delta(K_X)=0$. Thanks to \cite[Lemma 4.1]{Harris-Gross} we know that for any $f\in \R (X)^*$ the principal divisor div$(f)$ has an even number of points on each connected component of $X(\R)$. This allows us to extend this definition further to real linear systems. Given a real $g_d^r$, any connected component contributing to $\delta(g_d^r)$ is called a \textit{pseudo-line} for this $g_d^r$.
 \begin{example}
 	Consider a plane curve $X\subset \mathbb{P}^2_{\R}$ of degree $d$. Traditionally, a connected component of $X(\R)$ which intersects with a line in an odd number of points is called a \textit{pseudo-line} of $X$. Note that any effective divisor from the linear system $|\mathcal{O}_X(1)|$ is cut out on $X$ by the lines in $\mathbb{P}^2$. According to the previous definition, a connected component is a pseudo-line of $X$ if and only if it is a pseudo-line for $|\mathcal{O}_X(1)|$.
 	
 \end{example}
\begin{lemma}\label{deltaineq}
$\delta(g_d^r)\leq \text{min}\{n(X),d\}\text{ and }\delta(g_d^r)\equiv d \text{ mod 2}.$

\end{lemma}
\begin{proof}

	 Say $C_1,...,C_{n(X)}$ are the connected components of $X(\R)$ and take any $D\in g_d^r$. Let us write $P_{C_i}(D)=0$  if $\deg_{C_i}(D)$ is even and $P_{C_i}(D)=1$ otherwise. Then we have the relations \begin{align*}
	 	\delta(D)=\displaystyle \sum_{i=1}^{n(X)} P_{C_i}(D)\leq &\displaystyle \sum_{i=1}^{n(X)} 1=n(X),\\
	 \delta(D)=\displaystyle \sum_{i=1}^{n(X)} P_{C_i}(D)	\leq& \displaystyle \sum_{i=1}^{n(X)} deg_{C_i}(D) \leq\deg(D)=d.		
	 \end{align*}
		Note that $\deg_{C_i}(D)-P_{C_i}(D)$ is by definition an even number for each $i$.  As a non-real point appears with its conjugate in the support of the real divisor $D$, this yields
		\begin{align*}
			\delta(D)=\displaystyle \sum_{i=1}^{n(X)} P_{C_i}(D)&\equiv\displaystyle \sum_{i=1}^{n(X)} \deg_{C_i}(D) \hspace{0.4cm}(\textit{mod }2)\\
			&\equiv d\hspace{0.4cm} (\textit{mod }2).
		\end{align*}
\end{proof}

\subsection{Brill-Noether Theory}

Let $X^d$ be \textit{the d-fold product} of $X$ and $X_d:=X^d/S_d$ be \textit{the d-fold symmetric product }of $X$, where $S_d$ is the symmetric group of all permutations on $d$ symbols. $X_d$ is a smooth, projective variety of dimension $d$ \cite[Page 236]{principles} and its points correspond to effective divisors of degree $d$ on $X_{\mathbb{C}}$. Let $s:X^d\longrightarrow X_d$ be the quotient map and consider the following map 
$$
u:X_d\longrightarrow\text{Pic}^d(X)$$
$$D\longrightarrow |D|,$$
where $Pic^d(X)$ is the subscheme of $Pic(X)$ containing degree $d$ linear systems. 

Let $W^r_d:=\WRDX$ be the subvariety of $Pic^d(X)$ parametrizing linear systems on $X_\C$ of projective dimension at least $r.$ For every $D\in X_d$, $h^0(X_{\C},\mathcal{O}_{X_{\C}}(D))\geq1 $ by the effectiveness of $D$ so that $r(|D|)\geq 0$. Conversely, any linear system $|D|$ such that $r(|D|)=h^0(X_{\C},\mathcal{O}_{X_{\C}}(D))-1\geq 0$ must contain an effective divisor. Hence we may write \begin{equation}\label{imageu}
	u(X_d)=W_d^0(X_\C)
\end{equation}

Note that both maps $s$ and $u$ are defined over $\mathbb{R}.$ This information allows us to write the equality $\sigma(|D|)=|\sigma(D)|$. In this case the real points of $\WRDX$ correspond to the real linear systems of dimension at least $r$.

	Let $C_1,C_2,...,C_{n(X)}$ be the connected components of $X(\R)$. We recall an important proposition, together with its proof, as we will use it later.

\begin{prop}\cite[Proposition 3.2]{Harris-Gross}\label{compPic}
	\begin{enumerate}
		\item  $n(X_d)=\displaystyle\sum_{s=0}^{\lfloor d/2 \rfloor} \binom{n(X)}{d-2s}.$
		\item If $n(X)>0$, then $n(W_d^0)=n(X_d).$
		
	\end{enumerate}
\end{prop}

\begin{proof}
Let us define $$Q=\{(p,\bar{p})\in X(\C)\times X(\C)\}$$
	
	Consider the connected sets $$U(i_1,...,i_k)=C_{i_1}\times C_{i_2}\times . . . \times C_{i_k}\times Q^s$$ which are mapped into $X_d(\R)$ via the map $s:X^d\rightarrow X_d$ if $d=k+2s$. Since it is the quotient map, it is enough to consider the indices $i_1,...,i_k$ in a non-decreasing order $i_1\leq i_2\leq... \leq i_k$ so that the totality of such sets will cover the image. The intersections occur only when we have an equality at any step, say $i_j=i_{j+1}$. In this case images of the sets $U(i_1,...,i_j,i_{j+1},...,i_k)$ and $U(i_1,...,i_{j-1},i_{j+2},...,i_k)$  will have an intersection point $p_{i_1}+p_{i_2}+...+p_{i_{j-1}}+p+p+p_{i_{j+2}}+...+p_{i_k}$ where $p_{i_l}\in C_{i_l}$ for each $l=1,...,j-1,j+1,...,k$ and $p\in C_{i_j}$.

	Hence in order to find $n(X_d)$ we need to count the number of distinct  indices in increasing order $(i_1,...,i_k)$ such that $d-k\geq 0$ is even and $1\leq i_j\leq n(X)$ for each $j$ between 1 and $k$ (together with an empty set standing for $U(0)=Q^{d/2}$ when $d$ is even). We obtain the following count:
	\begin{align*}
		n(X_d)=&\binom{n(X)}{d}+\binom{n(X)}{d-2}+...+\binom{n(X)}{d-2\lfloor d/2\rfloor }\\
		=&\sum_{s=0}^{\lfloor d/2\rfloor}\binom{n(X)}{d-2s}	
	\end{align*}

We know that the image of $u:X_d\longrightarrow Pic^d(X)$ is $W_d^0$ by \eqref{imageu}. Let us take any $|A|\in W_d^0(\R)$. Since the linear series is real, it must contain a real divisor, say $B$ with $u(B)=|B|=|A|$. This shows that we still have the surjectivity $u\big(X_d(\R)\big)=W_d^0(\R)$ in the level of $u_{\R}:X_d(\R)\longrightarrow Pic^d(X)(\R)$ so that $n(X_d)\geq n(W_d^0)$. 

Let us show that $u$ takes connected components to connected components. Say we have two real divisors $D,E\in X_d(\R)$ such that the images under $u$ coincide, i.e, $|D|=|E|$. Then $D=E+\text{div}(f)$ for some $f\in\R(X)^*$. 

As $\delta(\text{div}(f))=0,$ $D$ and $E$ have the same pseudo-lines. Notice by the above construction that the connected components of $X_d(\R)$ are seperated exactly according to the pseudo-lines of their points. Therefore we may conclude that $D$ and $E$ are on the same component of $X_d(\R)$. This proves the other inequality $n(X_d)\leq n(W_d^0)$.

\end{proof}

We should note that the first part of the above proof follows the proof of \cite[Proposition 3.2]{Harris-Gross} almost verbatim. Second part uses a shorter idea and it better aligns with the overall flow of the current work.

\begin{remark}\label{remark1}
	Let $\varphi:\text{Pic}(X)\rightarrow J(X)$ be the extension of the Abel map $X\rightarrow J(X)$. If we identify the variety $W_d^r$ with its image $\varphi(W_d^r)$ and use the above proof, we can denote the corresponding connected component of $W_d^0(\R)$ inside $J(X$) containing the image of $U(i_1,...,i_k)$ under $\varphi \circ u$ by $V(i_1,...,i_k)$ (and the one containing the image of $Q^{d/2}$ by $V(0)$ when $d$ is even). 
	
	Observe that $\delta(|D|)=k$ for each $\varphi(|D|)\in V(i_1,...,i_k)$ since any such effective divisor $D$ lies in the component containing the image of $C_{i_1}\times..\times C_{i_k}\times Q^{d-2k}$ and all the $i_j's$ are different from each other. Also given such a number $k$, by the above count there exist exactly $\displaystyle\binom{n(X)}{k}$ components of $W_d^0(\R)$ whose points have their $\delta$ equal to $k$.

\end{remark}
\subsection{Trigonal Curves}

Recall that a non-hyperelliptic curve $X$ is called \textit{trigonal} if it admits a degree 3 map $\pi:X\rightarrow\mathbb{P}_{\mathbb{C}}^1$. In this case let $L=\pi^*(\mathcal{O}(1))$ be the trigonal bundle and $T$ be the corresponding $g_3^1$. Suppose that $g\geq 5$,  then $T$ is unique. So, keeping the assumption that $X$ is real, $T$ is a real $g_3^1$ and $\pi$ is defined over $\R$. Using Lemma \ref{deltaineq} we have either $\delta(T)=1$ or $\delta(T)=3$. In the latter case it is well-known that $n(X)=3$ and $g$ is even \cite[Proposition 5.1]{Mon}. In either case assume that $C_1$ is a pseudo-line for $T$ and $C_2,...,C_{n(X)}$ are the other connected components of $X(\R)$.

\begin{lemma}\label{divisorsinT}
	Suppose that $\delta(T)=1$. If $n(X)>1$, then for any integer $i$ with $1< i\leq n(X)$ there exists a real divisor $D=p_i+q_i+\tilde{q}_i\in T$ such that $p_i\in C_1$ and $q_i,\tilde{q}_i\in C_i$. Conversely, any real effective divisor $D \in T$ is either of this form (equality $i=1$ is allowed) or contains a conjugate pair such that $ D=p+q+\overline{q}$ where $p\in C_1, q\in X(\mathbb{C})\backslash X(\mathbb{R})$.
\end{lemma}
\begin{proof}
	Take any point $q\in C_i\subset X(\R)$ for any $i\neq 1$. Since $\pi$ is defined over $\R$, $\pi(q)$ lies in $ \mathbb{P}^1(\R)$. Therefore the divisor $D=\pi^*(\pi(q))\in T$ is an effective, degree 3 real divisor. Any such divisor must contain a point $p\in C_1$ as $C_1$ is a pseudo-line for $T$. By the assumption $\delta(T)=1$ we have $D=p+q+\tilde{q}$ where $\tilde{q}\in C_i$. 
	
	Conversely, any divisor $D\in T$ must also contain a point $p'\in C_1$. So $D-p'$ is a real divisor of degree 2 such that $\delta(D-p')=0$. If $D=p'+q_1+q_2$, either we have $\overline{q_1}=q_2\in X(\mathbb{C})\backslash X(\mathbb{R})$ or $q_1$ and $q_2$ are on the same component of $X(\R).$
\end{proof}

We need the following theorem to make the necessary cohomological calculations.

\begin{teorem}\label{tricoh}\cite[Chapter 1]{MarSch}
	\begin{itemize}
		\item[(1)] $\pi_* K_X\cong\mathcal{O}_{\mathbb{P}^1}(a)\oplus \mathcal{O}_{\mathbb{P}^1}(m)\oplus \mathcal{O}_{\mathbb{P}^1}(-2),$ where $a$ and $m$ are the unique integers satisfying $a\geq m$, $3a\leq 2g-2$ and $a+m=g-2$ .
		\item[(2)]
		For each integer $k$, we have 
		\begin{align*}
		H^0(X,K_X\otimes L^k)\cong &H^0(\mathbb{P}^1,\pi_* K_X\otimes \mathcal{O}_\p(k))\\
		\cong &H^0(\p,\mathcal{O}_\p (a+k))\oplus H^0(\p,\mathcal{O}_\p (m+k))\oplus H^0(\p,\mathcal{O}_\p (k-2)).
	\end{align*}
	\end{itemize}
\end{teorem}
Here $m$ is called the \textit{Maroni invariant} of the curve $X$ and it satisfies the following inequalities:
\begin{equation}\label{scrltyp}
	0<\dfrac{g-4}{3}\leq m\leq \dfrac{g-2}{2},\hspace{0.5cm} a=g-2-m.
\end{equation}

%

\begin{remark}\label{modulitrigonal}
	Using the isomorphism classes of such curves on $Y\cong \mathbb{P}\bigr(\mathcal{O}_\p (a-m)\oplus\mathcal{O}_\p\bigl)$, it is known that the space of real trigonal curves of genus $g$ and Maroni invariant $m$ is $(g+2m+4)$-dimensional when $\frac{g-4}{3}\leq m<\frac{g-2}{2}$  and $(2g+1)$-dimensional when $m=\frac{g-2}{2}$ \cite{Modtri}.

\end{remark}

Now let us fix $s(d,r,g):=2(d-1)-g-3(r-1)$, $\kappa:=\varphi(K_X)$, $\tau:=\varphi(T)$, where $\varphi:\text{Pic}(X)\rightarrow J(X)$  is the Abel map, and define the following sets:

$$U_d^r=\left\{
\begin{array}{cc}
	r\tau+W_{d-3r}^0, \hspace{0.5cm}\text{if} &d-3r\geq 0\\
	\emptyset, &\text{otherwise}\\
	
\end{array}
\right\},
$$

$$V_d^r=\left\{
\begin{array}{cc}
	\kappa-\bigr((g-d+r-1)\tau+W_{s(d,r,g)}^0\bigl), \hspace{0.5cm}\text{if }& s(d,r,g)\geq 0\\
	\emptyset, \hspace{2 cm} &\text{ otherwise}\\
	
\end{array}
\right\}.
$$
Now we state a crucial theorem that explains the irreducible components of $W_d^r$ for trigonal curves.
\begin{teorem}\cite[Proposition 1]{MarSch}\label{Marsc}
	For $d<g$ and $r\geq1$, we have
	\begin{enumerate}
		\item[(1)] $W_d^r=U_d^r\cup V_d^r$.
		\item[(2)] If $U_d^r\neq \emptyset$, then $U_d^r$ is an irreducible component of $W_d^r(X)$.
		\item[(3)] Suppose that $V_d^r\neq\emptyset$. Then also $U_d^r\neq\emptyset$ and $V_d^r$ is an irreducible component of $W_d^r$ different from $U_d^r$ if and only if $g-d+r-1\leq m$.
	\end{enumerate}
\end{teorem}
%
%

\section{Main Results}

For this chapter assume that $X, T, \pi $ and $m$ are as in the previous section. We want to calculate $n(W_d^r)$. If one of $U_d^r$ or $V_d^r$ is empty, then there is nothing to do as the other corresponds to a translate of $W_{d'}^0$ for some $d'>0$ and we know $n(W_{d'}^0)$ very well by Proposition \ref{compPic}. So by Theorem \ref{Marsc} we are interested in the case $g-d+r-1\leq m$ and $s(d,r,g)\geq 0.$

\begin{lemma}\label{intlemma}
	If \hspace{0.2cm}$0<g-d+r-1\leq m,$ then  $$	U_d^r(\R)\cap V_d^r(\R)\neq\emptyset\implies m\leq d-2r-1$$
\end{lemma}
\begin{proof}
		First note that degree of $K_X-(g-d+2r-1)T$ is $2g-2-3(g-d+2r-1)=3d-g-6r+1$. Since $a\geq m$ by Theorem \ref{tricoh} we obtain 	$h^0(X,K_X-(g-d+2r-1)T)\geq h^0(\p,\mathcal{O}(a-g+d-2r+1))$ so that
\begin{equation}\label{md2r-1}
		\begin{split}
		h^0(X,K_X-(g-d+2r-1)T)>0 &\iff a-(g-d+2r-1)\geq 0\\
		&\iff m\leq d-2r-1.
	\end{split}
\end{equation}

	Using the descriptions of these components we have
	\begin{equation}\label{D1+D2=D3}
\begin{split}
			U_d^r(\R)\cap V_d^r(\R)\neq\emptyset\iff& rT+D_1\sim K_X-(g-d+r-1)T-D_2\\
		&\text{ for some } D_1\in X_{d-3r}(\R) \text{ and }D_2\in X_{2(d-1)-g-3(r-1)}(\R)\\
\iff & K_X-(g-d+2r-1)T\sim D_1+D_2\\
\end{split}
	\end{equation} 
If this happens then $h^0(X,K_X-(g-d+2r-1)T)>0$, i.e, $m\leq d-2r-1$ by the above calculations.

\end{proof}
\begin{remark}
	Notice that the converse direction in the proof of Lemma \ref{intlemma} is correct if we can find a real divisor $D\sim K_X-(g-d+2r-1)T$ such that $D=D_1+D_2$ for some $D_1\in X_{d-3r}(\mathbb{R})$ and $D_2\in X_{2d-g-3r+1}(\mathbb{R})$.  In most cases this will be possible and we will use this lemma in both ways. However, if for example $d-3r\not\equiv\deg(D)\equiv 0 \hspace{0.3cm}\text {mod }2$ and $D$ does not contain any real points then it may not be possible to decompose D in this way.
\end{remark}
\begin{lemma} \label{g-d=m}
	If $\text{ }0<g-d+r-1\leq m\leq d-2r-1,$ then
	$$\kappa-(g-d+2r-1)\tau\in U_{3d-g-6r+1}^{2d-g-3r}(\R)\iff g-d+r-1=m$$
\end{lemma}

\begin{proof}
Note that $\kappa-(g-d+2r-1)\tau$ lies in $W_{3d-g-6r+1}^{a-g+d-2r+1}(\R)$ by \eqref{md2r-1}. Also we may write by assumption that $a-g+d-2r+1=g-2-m-g+d-2r+1\geq g-2-(g-d+r-1)-g+d-2r+1=2d-g-3r$. Hence, $\kappa-(g-d+2r-1)
\tau\in W_{3d-g-6r+1}^{2d-g-3r}(\R).$ Notice that $$\kappa-(g-d+2r-1)\tau\in U_{3d-g-6r+1}^{2d-g-3r}(\R)$$ exactly when $K_X-(g-d+2r-1)T\sim(2d-g-3r)T+D$ for some real $D\geq 0$, i.e when $h^0(X,K_X-(g-d+2r-1+2d-g-3r)T))=h^0(X,K_X-(d-r-1))>0$. 
Using Theorem \ref{tricoh} again, we have
\begin{equation*}
\begin{split}
		h^0(X,K_X-(d-r-1)T)=&h^0(\p,\mathcal{O}_\p (a-(d-r-1)))\\
		&\hspace{0.5cm}+ h^0(\p,\mathcal{O}_\p (m-(d-r-1)))\\
		&\hspace{0.5 cm} +h^0(\p,\mathcal{O}_\p (-2-(d-r-1)))\\
\end{split}
\end{equation*} 
Since $m\leq d-2r-1$ by assumption, we obtain $-2-(d-r-1)<m-(d-r-1)\leq d-2r-1-(d-r-1)=-r<0$ so that 
$$	h^0(X,K_X-(d-r-1)T)=h^0(\p,\mathcal{O}_\p (a-(d-r-1))).$$

Summarizing, we may write \begin{equation*}
\begin{split}
		\kappa-(g-d+1)\tau\in U_{3d-g-5}^{2d-g-3r}(\R)\iff&a-d+r+1\geq 0. \\
		\iff& g-d+r-1\geq m
\end{split}
\end{equation*} Under the hypothesis $g-d+r-1\leq m$ this is equivalent to saying that $g-d+r-1=m$.
\end{proof}

\begin{remark}\label{g-kt-basepts}
	So, under the assumptions $g-d+r-1=m\leq d-2r-1$ we have \begin{equation*}
	K_X-(g-d+2r-1)T\sim (2d-g-3r)T+(2g-3d+3r+1)\text{-many base points}.
\end{equation*}
 Composing with \eqref{D1+D2=D3}, in order to find the intersecting components of $U_d^r(\R)$ and $V_d^r(\R)$ it is enough to keep track of effective real divisors belonging to the  class $(2d-g-3r)T$ as long as $2g-3d+3r+1$ is small. 
\end{remark}
\begin{lemma}\label{deltanot3}
	If $g-d+r-1=m\leq d-2r-1$ and $2g-3d+3r+1=1$, then $\delta(T)=1$.
\end{lemma}
\begin{proof}
Since the number of base points in $K_X-(g-d+2r-1)T$ is $2g-3d+3r+1=1,$  we have the equivalence $$K_X-(g-d+2r-1)T\sim(2d-g-3r)T+p$$ for some fixed $p\in X$. As $K_X$ and $T$ are real, $p$ is in $X(\R)$. Note by the assumption $2g-3d+3r=0$ that $d$ and $r$ have the same parity. Then we may write $$g-d+2r-1\not\equiv 2d-g-3r \hspace{0.3cm}(\text{mod } 2).$$
Together with this, $\delta(K_X)=0$ allows us to write
	
	\begin{equation*}	\delta\left(K_X-(g-d+2r-1)T\right)=
		\begin{cases}
			\delta(T),& \text{if } g+d \text{ is even}\\
			0,& \text{if } g+d \text{ is odd}\\
		\end{cases}
	\end{equation*}
	
	and 
	\begin{equation*}	\delta((2d-g-3r)T)=
		\begin{cases}
			\delta(T),& \text{if } g+d \text{ is odd}\\
			0,& \text{if } g+d \text{ is even.}\\
		\end{cases}
	\end{equation*}

	Since adding one point to a divisor can change its $\delta$ by $+1$ or $-1$, if $\delta(T)=3$ the equivalence $K_X-(g-d+2r-1)T\sim (2d-g-3r)T +p$ creates a contradiction.
\end{proof}

	Recall  by Remark \ref{remark1} that the connected components of $U_d^r(\R)=r\tau+W_{d-3r}^0(\R)$ are of the form $$\tau + V(i_1,...,i_k),$$ where $1\leq i_1<i_2<...i_k\leq n(X)$, $d-3r-k\geq0$ and $d-3r\equiv k$ mod $2$. Similarly the components of $V_d^r(\R)=\kappa-\big((g-d+r-1)\tau+W_{s(g,r,d)}^0(\R)\big)$ are  $$\kappa-\bigl((g-d+r-1)\tau +V'(j_1,...,j_l)\bigr)$$

for each $1\leq j_1<j_2<...<j_l\leq n(X)$, $s(d,r,g)-l\geq 0$ and $s(d,r,g)\equiv l$ mod $2$,  where $s(d,r,g)=2(d-1)-g-3(r-1).$ (We use the notation $V'$ to avoid the confusion with the connected components of $U^r_d(\R)$) 

Since we have the equalities $W_d^r(\R)=U_d^r(\R)\cup V_d^r(\R)$, $n(U_d^r)=n(W_{d-3r}^0)$ and $n(V_d^r)=n(W_{s(d,r,g)}^0)$, we obtain the immediate bounds
\begin{align}\label{nwrd}
	0\leq n(W_d^r)\leq \sum_{s=0}^{\lfloor \frac{d-3r}{2}\rfloor}\binom{n(X)}{d-3r-2s} + \sum_{s=0}^{\lfloor \frac{s(d,r,g)}{2}\rfloor}\binom{n(X)}{s(d,r,g)-2s}.
\end{align}

\begin{lemma}\label{nocompofU}
	No component of $U_d^r(\R)$ can intersect with two distinct components of $V_d^r(\R)$ and vice versa.
\end{lemma}
\begin{proof}
	Assume to the contrary that $\tau+V(i_1,...,i_k)$ intersects simultaneously with $\kappa-\bigl((g-d+r-1)\tau+V'(j_{1_1},...,j_{l_1})\bigr)$ and $\kappa-\bigl((g-d+r-1)\tau +V'(j_{1_2},...,j_{l_2})\bigr)$. Then there exist effective real divisors $D_1,D_2,D_3$ and $D_4$ with $D_1,D_3\in U(i_1,...,i_k)$, $D_2\in U(j_{1_1},...,j_{l_1}), D_4\in U(j_{1_2},...,j_{l_2})$ such that  $$D_1+D_2\sim D_3+D_4\sim K_X-(g-d+2r-1)T$$
by \eqref{D1+D2=D3}. It yields $$D_1-D_3\sim D_4-D_2.$$
	Since $D_1$ and $D_3$ are on the same component $U(i_1,...,i_k)$ of $X_{d-3r}(\R)$, their pseudo-lines are exactly the same. This implies $$\delta(D_1-D_3)=0=\delta(D_4-D_2)$$ so that $\delta(D_4)=\delta(D_2).$ Indeed, $D_4$ and $D_2$ share the same pseudo-lines, i.e, they are on the same component of $X_{s(d,r,g)}(\R)=X_{2d-g-3r+1}(\R)$. Hence, we get $$j_{1_1}=j_{1_2},...,j_{l_1}=j_{l_2}$$
	
	such that two components $\kappa-\bigl((g-d+r-1)\tau+V'(j_{1_1},...,j_{l_1})\bigr)$ and $\kappa-\bigl((g-d+r-1)\tau +V'(j_{1_2},...,j_{l_2})\bigr)$ coincide.
	
	Using the same method, we can say that no component of $V_d^r(\R)$ can intersect with two distinct components of $U_d^r(\R)$. \end{proof}
Using Lemma \ref{nocompofU} we can upgrade the lower bound in \eqref{nwrd} for $n(W_d^r)$.
\begin{corollary}\label{cor1}
If $g-d+r-1\leq m$, then $$n(W_d^r)\geq\sum_{s=0}^{\lfloor \frac{d-3r}{2}\rfloor}\binom{n(X)}{d-3r-2s}.$$
\end{corollary}

\begin{proof}
By Lemma \ref{nocompofU} we may say that for each connected component $U$ of $U_d^r(\R)$ there can be at most one connected component $V$ of $V_d^r(\R)$ such that $V\cap U\neq\emptyset$ and $V$ has no non-empty intersection with other components of $U_d^r(\R)$. So, for each $U$ there is unique connected component of $W_d^r(\R)$ containing $U$. The same holds for the connected components of $V_d^r(\R)$ as well. This proves the formula
	\begin{align*}
		n(W_d^r)\geq \max\{n(U_d^r),n(V_d^r)\}.
	\end{align*}
	Since $g-d$ is positive, we have  $d-3r\geq 2d-g-3r+1=s(d,r,g)$. This shows that $$n(U_d^r)=\sum_{s=0}^{\lfloor \frac{d-3r}{2}\rfloor}\binom{n(X)}{d-3r-2s}\geq \sum_{s=0}^{\lfloor \frac{s(d,r,g)}{2}\rfloor}\binom{n(X)}{s(d,r,g)-2s}=n(V_d^r).$$
	
	The claim then follows.	 \end{proof}

The inequalities found in \eqref{nwrd} could be likely related to the intrinsic geometry of the curve itself. Let us see the following particular example.

\begin{example}
	Consider a trigonal curve $X$ of genus 5. We have $m=1$ by \eqref{scrltyp}. We will calculate $n(W_4^1)$. We have the bounds given in \eqref{nwrd} and Corollary \ref{cor1} that yield $$n(X)\leq n(W_4^1)\leq 2n(X).$$
	Take any real $T\in g_3^1$. Then by the Riemann-Roch formula, $$h^0(X_{\C},K_X(-T))=h^0(X_{\C},\mathcal{O}(T))+g-\deg(T)-1=2+5-3-1=3$$ so that $K_X-T$ defines a real $g_5^2$. Note that this $g_5^2$ cannot have a base point, otherwise we would get a $g_4^2$ but this contradicts with Clifford's theorem \cite[Page 107]{123} since a trigonal curve of genus $g\geq5$ cannot possess a hyperelliptic linear system (a  $g_2^1$)  by Theorem \ref{Marsc}.
	
	 In this case $K_X-T$ defines a real map $$\phi:X\longrightarrow \mathbb{P}_{\C}^2$$ such that the image $\phi(X)\in\mathbb{P}_{\C}^2$ is a plane curve of degree 5. The equality $g=5=\frac{(5-1)(5-2)}{2}-1$ implies by the degree-genus formula that $\phi(X)$ has only one singularity, a cusp or a node \cite[Page 208]{123}. This means that there exists exactly one pair of points $p_0,q_0\in X$ such that $$r(K_X-T-p_0-q_0)=r(K_X-T)-1=1.$$

	Since $g^1_3$ is unique and $\deg(K_X-T-p_0-q_0)=3$, we have $K_X-T-p_0-q_0\sim T$, i.e, $K_X-2T\sim p_0+q_0$. As this pair is the unique one such and both $K_X$ and $T$ are real, the divisor $p_0+q_0$ is a real divisor. Moreover, $\delta(p_0+q_0)=\delta(K_X-2T)=0.$ 
	
	On the other hand, $W_4^1$ has two irreducible components $U_4^1=\tau+W_1^0$ and $V_4^1=\kappa-(\tau+W_1^0)$ both of which are isomorphic to $X$.  Using the same method applied in \eqref{D1+D2=D3}, two connected components of $U^1_4(\R)$ and $V_4^1(\R)$ intersect if and only if there are points $p,q\in X(\R)$ such that $T+p\sim K_X-T-q$. But one can see that there are only two pairs satisfying such a linear equivalence. We may get at most two possible intersection points from the choices $(p,q)=(p_0,q_0)$ and $(p,q)=(q_0,p_0)$ depending on the reality of the points $p_0$ and $q_0$.
	
	Since $g$ is odd, $\delta(T)=1$ in this case. We break the rest into two cases:
	
	If the singularity above is a cusp or real node then $p_0,q_0\in X(\R)$ so that  $U_4^1(\R)\cap V_4^1(\R)\neq \emptyset$. 
	Since $\delta(p_0+q_0)=0,$ we have $p_0,q_0\in C_i$ for some $i=1,...,n(X).$ Then $\tau +V(i)$  intersects with $\kappa-(\tau +V'(i))$. In this case we get $n(W_4^1)=2n(X)-1$.

	If the singularity of $\phi(X)$ is a complex node, then $p_0=\overline{q_0}\notin X(\R)$ and $U_4^1(\R)\cap V_d^1(\R)=\emptyset$.  Hence, $n(W_4^1)=2n(X).$ 
\end{example}

We also want to improve the upper bound in \eqref{nwrd}. Suppose that the conditions $g-d+r-1=m$ and $m<d-2r-1$ hold. Then we can use Remark \ref{g-kt-basepts}.
Our strategy is to find different components of $X_{d-3r}(\R)$ and $X_{2d-g-3r+1}(\R)$ on which there are divisors $D_1$ and $D_2$, respectively, such that $$rT+D_1\sim K_X-(g-d+r-1)T-D_2,$$ equivalently, 
\begin{align}\label{D_1+D_2simT}
	D_1+D_2\sim K_X-(g-d+2r-1)T\sim (2d-g-3r)T+D_3
\end{align}
for some fixed $D_3\in X_{2g-3d+3r+1}(\R)$. Note that the value of $2g-3d+3r+1$ is always non-negative due to the bound $\frac{g-4}{3}\leq m$ given in \eqref{scrltyp}:
\begin{align*}
	2g-3d+3r+1=2g-3(g-m+r-1)+3r+1\geq 2g-3\left(g-\frac{g-4}{3}\right)+4=0.
\end{align*}

%

\begin{teorem}\label{teoremim}
\teotext
\end{teorem}
\begin{proof}
	The lower bound is already given in Corollary \ref{cor1}.	For the upper bound, let us first consider the equality $s(d,r,g)=2d-g-3r+1=0$. In this case $V_d^r$ is just one point and \eqref{nwrd} gives us an upper bound $$n(W_d^r)\leq s_{n(X)}(d-3r)+1$$ which is consistent with the statement of the theorem. Then we can assume $s(d,r,g)>0$ and divide the rest of the proof into three cases.
	
	\textbf{Case 1: } $\delta(T)=1\text{ and }n(X)=1.$
	
	In this case we know that both of $U_d^r(\R)$ and $V_d^r(\R)$ are connected, because of the equalities $$n(U_d^r)=\sum_{s=0}^{\lfloor \frac{d-3r}{2}\rfloor}\binom{1}{d-3r-2s}=1=\sum_{s=0}^{\lfloor \frac{s(d,r,g)}{2}\rfloor}\binom{1}{s(d,r,g)-2s}=n(V_d^r).$$
	
Since $\delta(T)=1,$ we can find points $q_1\in C_1$ and $q_2,q_3\in X$ such that $q_2+q_3$ is also real and $T\sim q_1+q_2+q_3$. If we choose the real effective divisors$$D_2=(2d-g-3r-1)q_1+q_2+q_3$$ and $$D_1=(2d-g-3r)(q_1+q_2+q_3)+D_3-D_2,$$  they satisfy $D_1+D_2\sim(2d-g-3r)T+D_3$,
where $D_3$ is the fixed real divisor in \eqref{D_1+D_2simT} such that the intersection $U_d^r(\R)\cap V_d^r(\R)$ is non-empty. Hence, $n(W_d^r)=1$ in this case and the theorem follows:

$$1=s_{n(X)}(d-3r)\leq n(W_d^r)\leq 1+\binom{n(X)-1}{2d-g-3r+1}.$$

	\textbf{Case 2: } $\delta(T)=1$ and $n(X)>1$.

	We want to show that any component $\kappa-\big((g-d+r-1)\tau+V'(j_1,...,j_l)\big)$ of $V_d^r(\R)$ has a non-empty intersection with $U_d^r(\R)$ possibly except the ones with $l=2d-g-3r+1$ and $j_1\neq1$. 
	
	Let us use the points $p_i\in C_1$ and $q_i,\tilde{q}_i\in C_i$ of Lemma \ref{divisorsinT} such that $T\sim p_i+q_i+\tilde{q}_i$ for each $i=2,...,n(X)$. 
	
	If there is such a connected component with $j_1\neq1$ and $l<2d-g-3r+1$, then $j_k>1$ for each $k$. Consider the equivalence
	$$(2d-g-3r)T+D_3\sim (2d-g-3r-l)(p_{j_l}+q_{j_l}+\tilde{q}_{j_l})+\sum_{n=j_1}^{j_l}(p_n+q_n+\tilde{q}_n)+D_3.$$
	Then the sum of the divisors
	\begin{equation*}
		\begin{split}
			D_1&=D_3+\left(\sum_{n=j_1}^{j_l}(p_n+\tilde{q}_n)\right)-p_{j_l}+(2d-g-3r-l)(q_{j_l}+\tilde{q}_{j_l}) \text{ and }\\
			D_2&=q_{j_1}+...+q_{j_l}+(2d-g-3r+1-l)p_{j_l}
		\end{split}
	\end{equation*}
	is linearly equivalent to $(2d-g-3)T+D_3$. Note that $D_2\in U(j_1,...,j_l)\subset X_{2d-g-2}(\R)$ since $2d-g-3r+1-l$ is an even number by the very description of these components. Then $\kappa-((g-d+r-1)\tau+V'(j_1,...,j_l))$ intersects with $U_d^r(\R)$.
	
	If $j_1=1$ and $l\leq2d-g-3r+1$, then we have another equivalence $$(2d-g-3r)T+D_3\sim (2d-g-3r+1-l)(p_{j_l}+q_{j_l}+\tilde{q}_{j_l})+\sum_{n=j_2}^{j_l}(p_n+q_n+\tilde{q}_n)+D_3$$ such that the divisors 	\begin{equation*}
		\begin{split}
			D_1&=D_3+\left(\sum_{n=j_2}^{j_l}(p_n+\tilde{q}_n)\right)-p_{j_l}+(2d-g-3r+1-l)(q_{j_l}+\tilde{q}_{j_l}),\\
			D_2&=q_{j_2}+...+q_{j_l}+(2d-g-3r+2-l)p_{j_l}
		\end{split}
	\end{equation*}
	add up to a divisor linearly equivalent to $(2d-g-3r)T+D_3$. As $2d-g-3r+2-l$ is an odd number, we have $D_2\in U(j_1,...,j_l)$.

	This shows that in addition to the number $n(U_d^r)=n(W^0_{d-3r})$, there could be at most $\binom{n(X)-1}{2d-g-3r-1}$-many possible connected components of the form 
	$$\kappa-\big((g-d)\tau+V'(j_1,...,j_{2d-g-3r+1})\big)$$ coming from $V_d^r(\R)$, where $2\leq j_1<...<j_{2d-g-3r+1}\leq n(X)$.
	
	\textbf{Case 3: }$\delta(T)=3$.
	

	In this case we know that $n(X)=3$ and $g$ is even. Without loss of generality, assume that $r$ is odd. This makes $2d-g-3r+1$ an even number so that the connected components of $V_d^r(\R)$ are just translates of $V'(0), V'(1,2), V'(1,3)$ and $V'(2,3)$. Since $\delta(T)=3,$ we have points $t_1\in C_1, t_2\in C_2, t_3\in C_3$ such that $t_1+t_2+t_3\sim T$. Consider the equivalence	
$$(2d-g-3r)T+D_3\sim (2d-g-3r)(t_1+t_2+t_3)+D_3.$$

	For each $i$ and $j$ with $1\leq i<j\leq3$ we can choose  $D_2=(2d-g-3r)t_i+t_j$ and $D_1$ as the rest of $(2d-g-3r)(t_1+t_2+t_3)+D_3$ such that $D_1+D_2\sim(2d-g-3r)T+D_3$. This choice gives rise to a point of intersection of $U_d^r(\R)$ and $\kappa-((g-d+r-1)\tau+V'(i,j))$.
	
	Also the choice $D_2=(2d-g-3r-1)t_1+2t_2$ lies in $U(0)$ and together with $D_1=t_1+(2d-g-3r-2)t_2+(2d-g-3r)t_3+D_3$ they add up to $(2d-g-3r)(t_1+t_2+t_3)+D_3.$ So, $\kappa-((g-d)\tau+V'(0))$ intersects with $U_d^r(\R)$ as well. Hence, we obtain the equality $n(W_d^r)=n(W^0_{d-3r})=s_{n(X)}(d-3r)$ in the case $\delta(T)=3$. \end{proof}

\begin{figure}
	\centering
	\includegraphics[width=1\linewidth]{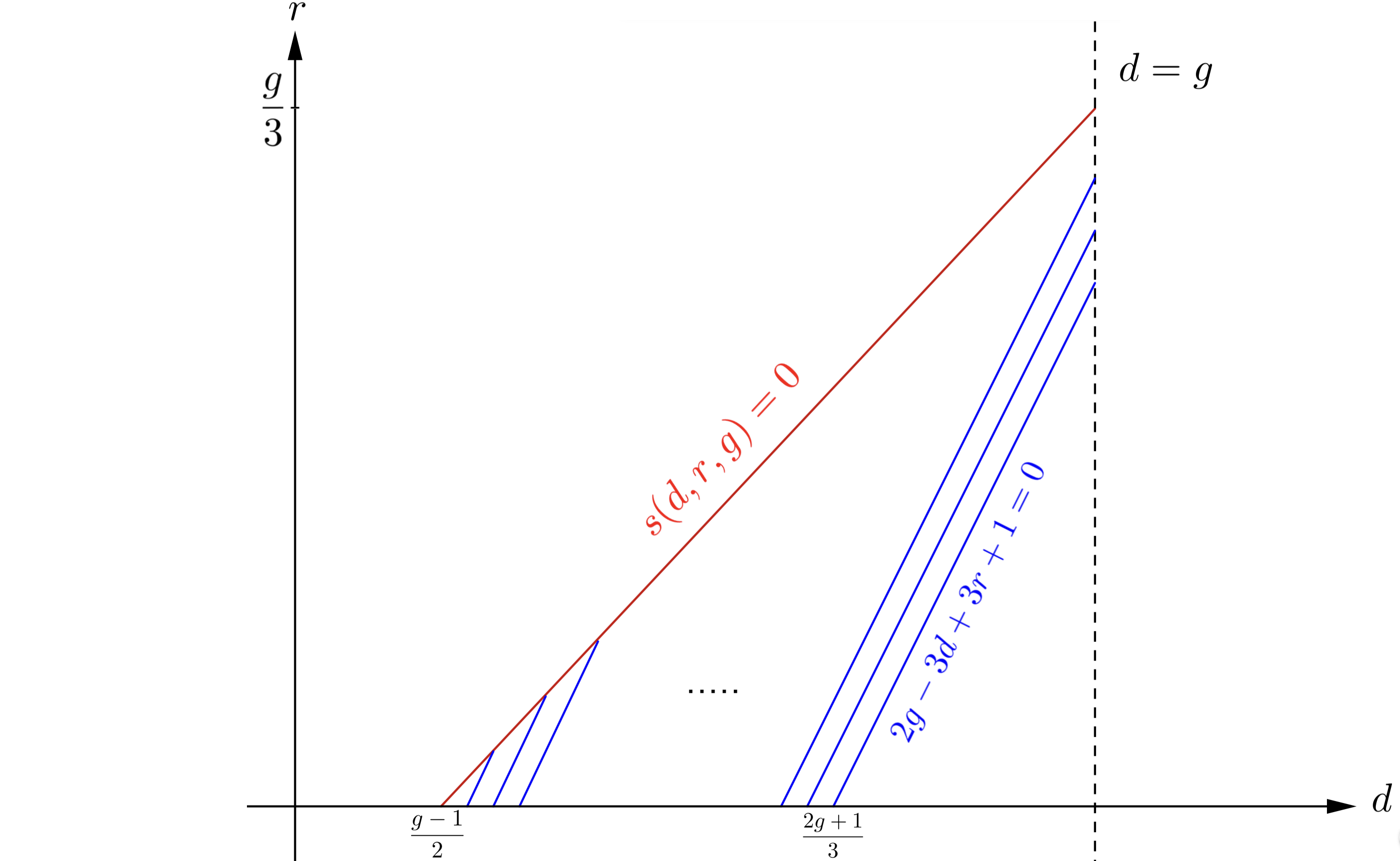}
	\caption{Admissible $(d,r)$ pairs}
	\label{fig:dr}
\end{figure}

Given the integers $g$ and $m$ satisfying $0<(g-4)/3\leq m\leq (g-2)/2$, we know the dimension of real trigonal curves of genus $g$ and Maroni invariant $m$ by Remark \ref{modulitrigonal}. At this point we may ask whether we have a pair $(d,r)$ satisfying the conditions of Theorem \ref{teoremim}. For such integers $g$ and $m$, there is a unique line segment in Figure \ref{fig:dr} depending on the value of $m$ parallel to the line $2g-3d+3r+1=0$ on which all the lattice points $(d,r)$ satisfy the conditons of Theorem \ref{teoremim}. Hence, this theorem upgrades the natural bounds obtained in \eqref{nwrd} and Corollary \ref{cor1} for $(d,r)$ values lying on a line segment.

If we consider the curves with $m=\dfrac{g-3}{3}$, then this means that we choose to look at the line segment $2g-3d+3r=0$ and in this case we can say more.

\begin{teorem}\label{teoremim2}
\thmtext

\end{teorem}

\begin{proof}
	Note that by Lemma \ref{deltanot3} $\delta(T)=1$ in this case. Moreover, if $C_1$ is the pseudo-line for $T$, then we have $$K_X-(g-d+2r-1)T\sim(2d-g-3r)T+p$$ for some fixed point $p\in C_1$ by Remark \ref{g-kt-basepts}. Suppose that an exceptional component $\kappa-\big((g-d+r-1)\tau+V'(j_1,...,j_l)\big)$ such that $l=2d-g-3r+1$ and $j_1\neq1$ has a non-empty intersection with $U_d^r(\R)$. Then there exist real effective divisors $D_1$ and $D_2$ such that $D_2\in U(j_1,...,j_l)\subset X_{2d-g-3r+1}(\R)$ and $$D_1+D_2\sim\left(2d-g-3r\right)T+p. $$
	
	Recall that $p$ is the base point of this linear system. Since $\delta(D_2)=\deg(D_2)=2d-g-3r+1$ and $j_1\neq1,$ $D_2$ cannot contain $p$. So, $D:=D_1-p$ is effective.  On the other hand, effective divisors linearly equivalent to $T$ cannot contain two points with different pseudo-lines other than $C_1$. Therefore, there cannot be a divisor $D_2$ with $\delta(D_2)=2d-g-3r+1$ such that $C_1$ is not a pseudo-line for $D_2$ and $D+D_2\sim (2d-g-3r)T$ for some $D>0$. 
\end{proof}

\section*{Acknowledgements}
I would like to thank my supervisor Prof. Ali Ulaş Özgür Kişisel for his sincere guidance and supportive feedback throughout my PhD research. This study builds upon the findings presented in my doctoral dissertation \cite{akyar}.

\end{document}